\documentclass[12pt]{amsart}

\usepackage{amssymb}
\usepackage{graphicx}
\usepackage[mathscr]{eucal}

\usepackage{enumerate, verbatim,enumitem}
\usepackage{amscd,mathabx}

\newtheorem{theorem}{Theorem}[section]
\newtheorem{lemma}[theorem]{Lemma}

\theoremstyle{definition}
\newtheorem{definition}[theorem]{Definition}

\newtheorem{corollary}{Corollary}
\newtheorem{proposition}{Proposition}

\theoremstyle{remark}
\newtheorem{remark}[theorem]{Remark}

\newcommand{\R}{\mathbb R}
\newcommand{\C}{\mathbb C}
\newcommand{\N}{\mathbb N}
\newcommand{\Q}{\mathbb Q}
\newcommand{\Z}{\mathbb Z}

\newcommand{\ci}{\widehat{\mathbb C}}
\newcommand{\ds}{\displaystyle}
\newcommand{\orb}{\mathcal O}

\newcommand{\LA}{\Lambda}

\newcommand{\pl}{\wp_{\Lambda}}

\numberwithin{equation}{section}

\begin{document}

\title[Iterated $\pl$ integer points]{Integer points for iterated Weierstrass elliptic functions}

\author[Hawkins]{Jane  Hawkins }
\address{ Department of Mathematics \\ University of North Carolina
at Chapel Hill\\ CB \#3250 \\ Chapel Hill, North Carolina
27599-3250}
 \email{jmh@math.unc.edu}


\date{}

\begin{abstract}
Iterated Weierstrass $\wp$ functions exhibit relationships between period lattices and the dynamics.
We connect the dynamical properties
 to the occurrence of integer and rational points on associated real elliptic curves. Defining $z \in \C$ to be an integer point if both $\pl(z)$
and $\pl'(z)$ are integers, we restrict to $\pl$ when the lattice $\LA$ is real and its invariants  $g_2$ and $g_3$ are integers.
We study when an integer point lies in the Fatou set of $\pl$ and when it
lies in the Julia set, and if there is any dynamical significance to this.
Applying established results on integer solutions to cubic Diophantine equations to the dynamical properties of iterated elliptic $\wp$ functions, we show
integer points at poles of $\pl$ lie in the Julia set. Integer points can also occur in the Fatou set of $\pl$.
\end{abstract}

\maketitle

 \section{Preliminary discussion}

   An elliptic function is a doubly periodic meromorphic function in $\C$ with a period lattice $\Lambda$.   The dynamics of iterated elliptic functions fit into the landscape of complex dynamics of meromorphic functions, worked on by many investigators, e.g., \cite{bak,ber,dev1,hk1,koturb,rip}.
There are elliptic curves associated to the Weierstrass $\wp$ functions, namely the Lie group $\C/\Lambda$ where $\Lambda$ is a two-dimensional lattice
in $\C$.  If the lattice $\Lambda$ is real, that is, if $\overline{\LA}=\LA$, where $\overline{\
\LA}$ denotes the complex conjugate of each lattice point in $\LA$, then
by restricting $\pl$ to the real line, we have a periodic function. For certain lattices, in addition there are non-real points $z \in \C$ with real values for $\pl(z)$ and $\pl'(z)$, and define a real elliptic curve.
The purpose of this paper is to explore some connections between known dynamical properties of the iterated elliptic $\wp$ function and its associated Diophantine equation that defines a real elliptic curve.   Our point of view is different from the classical one that views the Weierstrass $\pl$ function as a 2-fold branched cover of the Riemann sphere  $\ci$  by $\C/\LA$.  From that perspective, Latt\`{e}s rational maps arise; rational maps $R$ with $J(R) =\ci$ for every choice of lattice. However, iterating $\pl$ yields many types of Julia sets, which gives our starting point.

In this paper we view an integer point on an elliptic curve as a pair $(\pl(z), \pl'(z))$ lying on a real elliptic curve; some of its properties are determined by the dynamics of the iterated function $\pl$ relative to its period lattice.   We can divide the points on the curve into Julia or Fatou points according to whether or not the underlying point $z \in J(\pl)$.

\subsection{Background notation and results} \label{sec:background}
Elliptic functions are doubly periodic meromorphic functions, and lattices determine double periods for elliptic functions.
By $\Lambda=[\lambda_{1},\lambda_{2} ] $ we denote the group
\begin{equation}\label{eqn:latticedefn}  \Lambda=  \{m \lambda_{1} + n
\lambda_{2} \colon m,n \in {\mathbb Z} \} \subset \C. \end{equation}  Assuming $\lambda_1, \lambda_2 \in \C$
are non-zero and linearly independent over $\R$, we call $\Lambda$ a {\em lattice}.
A lattice $\Lambda$ acts on $\mathbb C$ by translation, each $\omega
\in \Lambda$ inducing the transformation $\ds T_\omega:z \mapsto z+\omega.$
We denote the coset of $\C/\Lambda$ containing
$z$ by $[z]$.
A closed, connected parallelogram $Q$ is a {\em fundamental
parallelogram} for $\Lambda$ if
  \begin{enumerate}
\item  for each $z \in \mathbb C$, $Q$ contains at least one point
in the same $\Lambda$-orbit as $z$; \item no two points in the
interior of $Q$ are in the same $\Lambda$-orbit.
\end{enumerate}
 Other regions are possible that satisfy (1) and (2).
 A {\em real lattice} $\LA$ is one such that $\overline{\LA}=\LA$, where $\overline{z}$ denotes complex conjugation
 and $\overline{\LA}=\{ \overline{\lambda}\,:\, \lambda \in \LA\}$.
Real lattices with two generators are of only two types. The first type is rectangular, where $\Lambda =[a, ib]$, with $a,b$ both real, the lattice points
 forming vertices of rectangles.

 For the second type, called rhombic, generators can be chosen to be a complex conjugate pair, so $\Lambda = [a+ ib, a-ib]$, and the lattices form the vertices
 of a tiling of the plane by rhombi with vertical and horizontal diagonals.

If $\Lambda$ is real rhombic, it is called {\em horizontal} if the longer diagonal is horizontal, and {\em vertical} if the longer diagonal is vertical.

\subsection{The Weierstass $\wp$ function}
The
{\em Weierstrass elliptic function} is defined for each $z \in {\mathbb C}$, by
\begin{equation} \label{eqn:wp}  \wp_{\Lambda}(z) = \frac{1}{z^{2}} + \sum_{\lambda \in \Lambda
\setminus \{0 \}} \left(\frac{1}{(z-\lambda)^{2}} - \frac{1}{\lambda^{2}}
\right). \end{equation}
The series in (\ref{eqn:wp}) converges absolutely and uniformly in every bounded region with no lattice points.
The map $\wp_{\Lambda}$ is an even elliptic function with poles of order 2 at lattice points.
 The derivative $\pl'$ is
an odd elliptic function of order 3, also periodic with respect to
$\Lambda$, and obtained by term by term differentiation of (\ref{eqn:wp}).
A critical point $z_0$ satisfies
$\wp_{\Lambda}'(z_0) = 0$ and $\wp_{\Lambda}(z_0)$ is a critical value.  While there are infinitely
many critical points occurring at half lattice points, there are only 3 critical values.
For each lattice $\Lambda = [\lambda_1,\lambda_2]$ and $k \in \C \setminus \{0\}$, there is a
nonlinear relation between $\wp_\Lambda$ and $\wp_{k \Lambda}.$
From the series definition of $\wp$, we have the following identities:
\begin{equation}\label{eqn:homog}  \wp_{k \Lambda}( kz) = \frac {1}{k^2} \pl(z), \, \,{\rm and}\, \,
\,  \wp_{k \Lambda}'( kz) = \frac {1}{k^3} \pl'(z). \end{equation}
The homogeneity conditions lead to widely varying dynamical behavior when iterating $\pl$.  For $\Lambda$ a real lattice,
the series also gives $\ds \pl(\overline{z})=\overline{\pl(z)},$
since $\overline{\Lambda}= \Lambda$.

Using the notation from (\ref{eqn:latticedefn}),  we define  $c_1 = \lambda_1/2$, $c_2 = \lambda_2/2$, and $c_3 = \lambda_1/2 +  \lambda_2/2$ to denote the half lattice points in a fundamental parallelogram $Q$.
This is sometimes ambiguous since generators can change; however
the critical points of $\pl$ are $[c_j]$ for $j = 1,2,3$, with
$$[c_j]=\{c_j + \lambda,\, \lambda \in \Lambda\}.$$
The critical values of $\pl$ are denoted by $e_j
 = \pl(v_j)$ for $j = 1,2,3,\, v_j \in [c_j]$. They are distinct for each lattice, and they satisfy the relationship
\begin{equation}
    \label{es}
    e_{1}+ e_{2} + e_{3}=0, \; \; \; \;
    e_{1}e_{3}+e_{2}e_{3}+e_{1}e_{2}= \frac{-g_{2}}{4}, \; \; \; \;
    e_{1}e_{2}  e_{3}=\frac{g_{3}}{4},
\end{equation}
where $g_2$ and $g_3$ are as follows.   Given a lattice $\LA$, define
$ \ds  S_n = \sum_{\lambda \in \LA, \lambda \neq 0} \lambda^{-n}, $
which converges absolutely for all $n \geq 3$.  We then set
$\ds  g_2 = 60 \, S_4,$ \,  \, $ \ds g_3 =  140\, S_6, $
so that
\begin{equation}\label{eqn:mainid}  \ds \wp'(z)^2 = 4 \pl(z)^3-g_2 \pl(z) - g_3. \end{equation}
Let $\bigtriangleup = g_2^3-27g_3^2$
denote the discriminant of the cubic polynomial from (\ref{eqn:mainid}):
\begin{equation}\label{eqn:cubicpoly} p(x) = 4 x^3 - g_2 x -g_3;\end{equation}
 it satisfies $\bigtriangleup \neq 0$ since the $e_j$s are distinct.  We see from (\ref{eqn:mainid}) and (\ref{eqn:cubicpoly})
 that  $\pl$ satisfies the differential equation $\ds \left(\pl'\right)^2=p(\pl)$.
A lattice $\Lambda$ is real if both $g_2,g_3 \in \R$, (this is equivalent to the statement
that $\overline{\Lambda}=\Lambda$).

Writing both $\Lambda(g_2,g_3)$ and $g_2(\lambda_1,\lambda_2)$ and $g_3(\lambda_1,\lambda_2)$  to show the dependence of the invariants on each other,
we have furthermore that
$ \ds g_2(k \lambda_1,k \lambda_2) =  k^{-4} g_2(\lambda_1,\lambda_2)$ and $g_3(k \lambda_1,k \lambda_2)= k^{-6} g_3(\lambda_1,\lambda_2).$

A few more classical identities help in the following material.  By differentiating (\ref{eqn:mainid}) with respect to $z$,
and simplifying, we see that
\begin{equation}\label{eqn:pprimeprime} \pl''(z) = 6 (\pl(z)^2)  - 1/2 g_2.\end{equation}  This leads to a ``point doubling identity'' that has the following form.  We set $x = \pl(z).$
\begin{equation} \label{eqn:double}  \pl( 2z) = \frac {x^4 + (1/2) g_2 x^2 +2 g_3 x + (1/16) g_2^2}{4 x^3 - g_2 x - g_3}. \end{equation}
A classical version of (\ref{eqn:double}) is
$$  \pl(2 z) = \frac 14 \left ( \frac{\pl''(z)}{\pl'(z)} \right )^2 - 2 \pl(z),$$
derived in many texts such as \cite{duv}.
In this paper we focus on lattices $\LA$ for which both $g_2$ and $g_3$ are integers.   Therefore, if $\pl(z) \in \Z$,
then the expression in (\ref{eqn:double}) is always rational and sometimes an integer.

More generally, it is useful to have this addition identity:
\begin{equation} \label{eqn:addition}  \pl( z + w) = \frac 14 \left (\frac {\pl'(z)-\pl'(w)}{\pl(z)-\pl(w)}\right)^2 - \pl(z)- \pl(w). \end{equation}

The {\it Fatou set} $F(\pl)$ is the
set of points $  z \in \ci$ such that $ \{\pl^{k} \colon k
\in \N \} $ is defined and normal in some neighborhood of $ z $. The
{\it Julia set} $J(\pl)$  is the complement of the Fatou set on the sphere.  For each
lattice $\Lambda$, every elliptic function with period lattice
$\Lambda$ is of Class $S$ \cite{ber} and does not have wandering domains
\cite{bak} or Baker domains \cite{ber}. Order 2 elliptic functions do not have Herman rings,
and if $\Lambda$ is a real rectangular lattice, then $\pl$ cannot have a Siegel disk \cite{hk1}.
Therefore in this setting,
$\pl$ has a short list of Fatou component types, which we summarize here.

\begin{theorem}\label{thm:fatoubound}  For a period lattice $\Lambda$,
every component  of $F(\wp_\Lambda)$ is preperiodic, and each
 forward
invariant Fatou  component
 contains one of the following:
\begin{enumerate}
\item a linearizing neighborhood of an attracting periodic point;
\item a B\"{o}ttcher neighborhood of a superattracting periodic
point; \item an attracting Leau petal for a periodic parabolic
point.  The periodic point is in $J(\wp_\Lambda);$
 \item a periodic Siegel
disk containing an irrationally neutral periodic point.
\end{enumerate}
\end{theorem}

There are unbounded Fatou components that are possible in this setting \cite{hk3}, see also Figures \ref{fig:toralband} and \ref{fig:rhombictoral}.

 \begin{definition} \label{def:toralband} Assume $\pl$ has period lattice $\Lambda=[\lambda_1, \lambda_2].$
A Fatou component $A_0$ of $F(\wp)$
is a {\it toral band} if $A_0$ contains an open subset $U$
which is simply connected in $\C$,  but $U$ projects to
a topological band around the torus $\C/\Lambda$  containing a
homotopically nontrivial curve.

\end{definition}

In parametrizing the dynamics of elliptic functions, we use the fact that
the values $g_{2}$ and $g_{3}$ are complete invariants
of the lattice $\Lambda$: if
$g_{2}(\Lambda) = g_{2}(\Lambda^\prime)$ and $g_{3}(\Lambda) =
g_{3}(\Lambda^\prime)$, then $\Lambda = \Lambda^\prime$.
Furthermore given any $g_{2}$ and $g_{3}$ such that
$g_{2}^{3}-27g_{3}^{2} \neq 0$ there exists a lattice $\Lambda$
having $g_{2}= g_2(\Lambda)$ and $g_{3}= g_3(\Lambda)$ as its
invariants \cite{duv}.
It has long been established that changing the invariants $g_2$ and $g_3$ changes the Julia and Fatou sets of $\pl$ \cite{hk1}, see also Figures
\ref{fig:fourgroup},\ref{fig:toralband},\ref{fig:Juliacurve},\ref{fig:per2}, and \ref{fig:rhombictoral}.

Under our hypotheses, we use techniques developed for dynamics to analyze lattices $\LA = \LA(g_2,g_3)$ for which there are points $z \in \C/\Lambda$ such that both $\pl(z)$ and $\pl'(z)$
are integers.  For example, under our hypotheses on $\LA$, by (\ref{eqn:mainid}), if $\pl(z)$ is an integer, then $\pl'(z)$ is either an integer or a square root of an integer.   Moreover since for every real lattice,  $\pl:\R \rightarrow [e_r, \infty]$ is surjective, for $e_r$ a real critical value, there are infinity many points in each fundamental period on $\R$ for which
$\pl(z) \in \Z$.   The goal of this paper then is to understand for which $z$ both $\pl(z)$ and $\pl'(z)$ are integers, and whether these points lie in the Fatou or Julia sets of $\pl$, or some of each.

There are many subtleties to the problem and they relate to the Diophantine equation determined by (\ref{eqn:mainid}). Before turning to that,
we integrate the additive group structure on the flat torus $\C/\Lambda$ into the discussion.   Our interest lies
in rational (and in particular, integer) values of $\pl(z)$ and $\pl'(z)$, for $z \in \C/\LA$, because the points in the plane on which both $\pl$ and $\pl'$ have rational values a form an additive group in the following sense.

\begin{proposition}\label{prop:toralgroup}  Assume that $g_2 \in \Z$ and $g_3 \in \Z$, and $\LA=\LA(g_2,g_3)$ is a lattice. If $[0]$ denotes the residue class of any $\lambda \in \LA$ (all poles for  $\pl$ are in this residue class), then the set
\begin{equation} \label{eqn:realelliptic} E_{(g_2,g_3)}(\Q)= \{ z \in \C/\Lambda \,\, {\rm such \, that} \, \, \pl(z),\pl'(z) \in \Q \} \end{equation}
forms a group under the usual addition on $\C/\LA$ given by $[z]+[v] = [z+v]$, with $[0]$ as the group identity.
\end{proposition}

\begin{proof}  We rewrite  $[z +w]$ as $z+w (\bmod  \LA)$, so that if $\lambda, \mu \in \LA$ then $(z+\lambda) +(w+\mu) (\bmod \LA)=[z+w].$
 Consider $E_{(g_2,g_3)}(\Q)$ as in (\ref{eqn:realelliptic}).  Suppose $z,w \in \C$ are of the form $z,w \in \R$;   then $z+w \in \R$ as well.  Therefore if $\LA$ is a real rhombic or rectangular lattice, then $z,w \in E_{(g_2,g_3)}(\Q)$ implies that $\pl(z), \pl'(z),  \pl(w), \pl'(w) \in \Q$,  and  $\pl(z+w), \pl'(z+w) \in \Q$ as well,  by  (\ref{eqn:addition}).
If $\LA$ is rhombic, this is all we need to show for the moment.
If $\LA$ is rectangular, we write $\LA=[a, bi], \, a,b \in \R$, and we label $c_1 = a/2, c_2 = i b/2$, and $c_3 = (a+ i b)/2.$

If both $z,w \in  E_{(g_2,g_3)}\in(\Q)$ are of the form $z=\zeta+ c_2, w = \omega+ c_2$, with  $\zeta,\omega \in \R$, then
$z+w = (\zeta + \omega) +  i b = \zeta+\omega (\bmod \LA)$, so the sum is a real number; therefore $\pl(z+w)$ and  $\pl'(z+w)$ are real too.  Moreover by  (\ref{eqn:addition}) $\pl(z+w), \pl'(z+w) \in \Q$ as well.
If exactly one of $z,w \in E_{(g_2,g_3)}\in (\Q)$ is of the form  $v=\nu+ c_2, \nu \in \R$, then label that point $z$.  Then $z+w = (\nu +w) + c_2$ and again (\ref{eqn:addition}) implies that $\pl(z+w), \pl'(z+w) \in \Q$ as well.    The additive negative of $z \in E_{(g_2,g_3)}(\Q)$ is $-z (\bmod \LA)$, so the identity element is $[0]$ as claimed earlier.
\end{proof}

 \subsection{Elliptic curves}
 We present only the material needed this paper, so while we extend the definition given in (\ref{eqn:realelliptic}), it is presented much more generally in other sources such as \cite{silvermantate}.
     Let  $K$  denote one of the following fields:  the rational numbers
$\Q$, the real numbers $\R$, or the complex numbers $\C$.
   Given $g_2,g_3 \in K$ such that $g_2^3-27 g_3^2 \neq 0$, a {\em (nonsingular) elliptic curve over $K$}, denoted $E_{(g_2,g_3)}(K)$, is defined by
   \begin{equation} \label{eqn:ellipticcurve} E_{(g_2,g_3)}(K)= \{ z \in \C/\Lambda \,\, {\rm such \, that} \, \, \pl(z),\pl'(z) \in K \}, \end{equation} the set of $K$-valued solutions to the equation  $\ds  y^2 = p(\pl)$ as well as a point $\orb$, with the added condition that there are distinct solutions to
   $p(x)$.  This is equivalent to saying the discriminant of the polynomial $p$ in (\ref{eqn:cubicpoly}) should be nonzero.

  Since $\Q \subset \R \subset \C$, we fix $g_2,g_3$ as above and usually write $E(K) = E_{(g_2,g_3)}(K)$. We define the sets
  \begin{equation}\label{eqn:JuliaFatou}  \widehat{J}(E(K)) = J (\pl) \cap (K \cup \{\infty\}) \quad {\rm{and}} \quad  \widehat{F}(E(K)) = F(\pl) \cap K. \end{equation}
By a slight extension of usual convention, we call $\widehat{J}(E(K))$ and  $\widehat{F}(E(K))$ the {\em Julia and Fatou sets} of $E(K)$.  We note that $\{\infty\} \notin F(\pl)$.

\begin{remark}
It follows that $ \widehat{J}(E(\C)) = J (\pl)$.  From (\cite{hk3}, Propositions 5.5 and 5.6) it often happens that  $\widehat{J}(E(\R))$ is a Cantor set while $J(\pl)$ is connected.
It is also possible that $\widehat{J}(E(\R)) = \R$ while $F(\pl) \neq \emptyset$ (\cite{hk3}, Theorem 5.2).

\end{remark}

 \begin{figure}
\begin{center}
{\includegraphics[width=.95\textwidth]{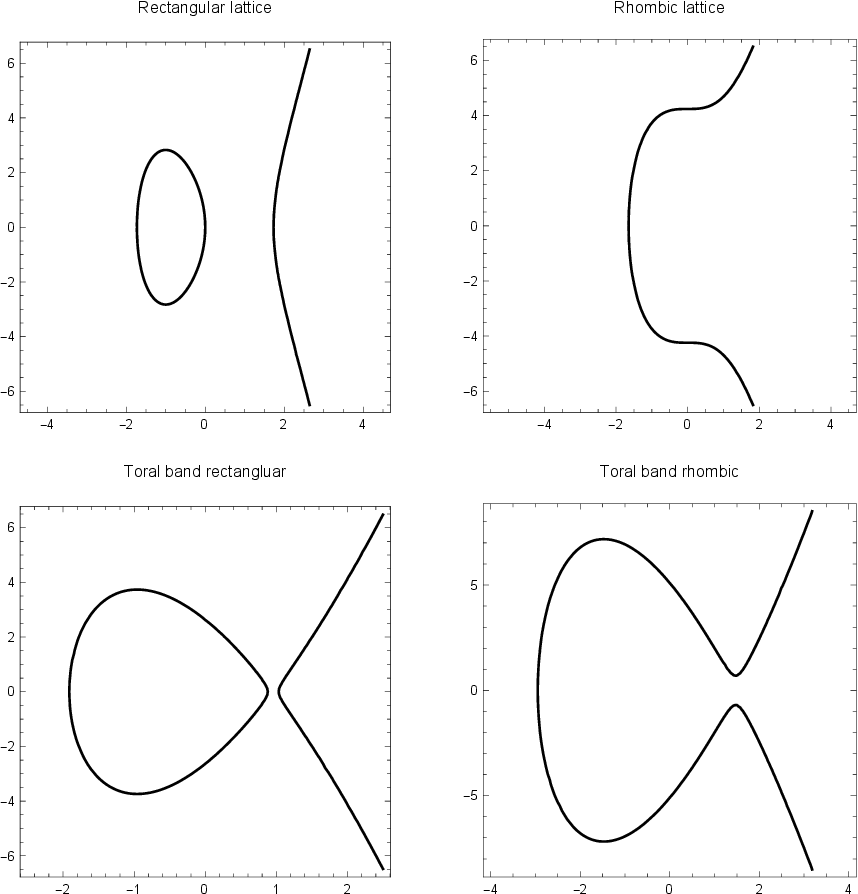}}
\caption{Every pair of real invariants $g_2,g_3$ gives rise to a lattice $\LA$ via (\ref{eqn:mainid}). The critical values of $\pl$ corresponds to points with $y$ coordinate $0$ on the curve,as well as complex conjugate points, and toral bands can occur in the Fatou set when two critical values are close to each other, as shown in the second row.} \label{fig:curve12}
\end{center}
 \end{figure}

Real elliptic curves are shown in Figure \ref{fig:curve12}.
An elliptic curve over $\R$ is an abelian group,  with a group addition operation defined on it and known to Fermat.  The additive identity is denoted by $\orb$, and corresponds to the point at $\infty$ in the plane, which, when added to the curve $E$ in $\R^2$ compactifies it.
If E is an elliptic curve, with points $A,B \in E$, then the chord joining $A$ and $B$ (or the tangent to the curve at $A$, if $A=B$) intersects $E$ at a unique third point called
$A*B$.
The reflection of $A*B$ about the  $x$-axis yields a unique point, which is defined to be  $A+B$; if $A = B$, we call this reflection $2A$.  When $B = -A$, then $A+B = \orb$.  It is well-established that this operation makes $E_{(g_2,g_3)}(\R)$ into an abelian group.  Moreover,  $E_{(g_2,g_3)}(\Q)$ is a subgroup and the addition operations yield the formulas in (\ref{eqn:double}) and (\ref{eqn:addition}).

   \subsubsection{Group structure.}
   The pairs of numbers of the form $(\pl(z),\pl'(z))$ with rational coordinates form a group, if we add an identity point $\orb$ (the point at  $\infty$
   on $\ci$, and any lattice point on $\C/\Lambda)$.  There is a geometric interpretation on the associated elliptic curve, and on the torus (or circle(s)), there is the usual group structure.
   
There are more than three ways to consider the group identity element for a real elliptic curve $E_{(g_2,g_3)}(\R)$.  On the complex torus $\C/\LA(g_2,g_3)$, the group identity element is the coset containing the origin $0+0i$;  since every $\lambda \in \LA$, satisfies $\lambda \in [0+0 i]$, every lattice point plays the role of the identity on $\C/\LA$.  Since we are restricting to real elliptic curves, the statement is similar, but we only consider the coset for $\R$-valued lattice points,  which means that we look at the coset of a real generator of $\Lambda.$

   If we consider $E_{(g_2,g_3)}(\R)$ as a one-dimensional curve in $\R^2$, the group identity element $\orb$ is identified with the point at $(\infty,\infty)$, and we compactify $E_{(g_2,g_3)}(\R)$ by adding this point and carrying the topology over from $\R / (\Lambda \cap \R)$.

 If $(x,y)$ is a rational point on an elliptic curve of finite order, then the  coordinates are integers, and there are only a finite number of such points.  They form an additive subgroup called the torsion group of $E_{(g_2,g_3)}(\Q)$, and usually written $E_{\rm tor}(\Q)$; it can be the trivial group $\{\orb\}$.
 An element of $E_{\rm tor}(\Q)$ is called a torsion point.
This result is called the Nagell-Lutz Theorem and a proof appears in \cite{silvermantate}.  A stronger result is that only a short list of torsion groups can occur,
 and the group $E_{(g_2,g_3)}(\Q)$ consists of a torsion group and a finitely generated group, (proved by Mordell in 1922).

 Throughout, by $\Z_n$ we denote the finite abelian group with $n$ elements and isomorphic to $\{0,1,\ldots,n-1\}$ with the operation of addition $(\bmod n)$.

\section{Dynamics of rectangular lattices with integer points}

Real rectangular lattices $\Lambda=[\lambda_1,\lambda_2 i], \, \lambda_1,\lambda_2 >0$ give real elliptic curves with two components as shown in Figure \ref{fig:curve12}.
A real lattice $\LA$ is square if and only if $g_3=0$. The next result is well-known (see \cite{silvermantate}, Thm. 2.1).
\begin{lemma} A lattice gives rise to an order 2 torsion point if and only if there exists an integer-valued critical value for $\pl$. \end{lemma}
\begin{proof}  The lemma follows from Proposition \ref{prop:toralgroup};  on the torus $\C/\LA$,
a point of order 2 needs to be a half lattice point, which is a critical point.  An order 2 point is in a finite abelian group only if the (integer) torsion group has even order. \end{proof}

 In the real rectangular lattice case, we can compute integer points in $E_{(g_2,g_3)}(\Q)$ using results from \cite{hk3}.  In iterated Weierstrass $\wp$ functions, when finding examples of lattices with various dynamical properties we often start with a {\em standard lattice}, which is a lattice with invariants $g_2, g_3=-g2+4$.  In this way we always obtain a critical value of $1$ or $-1$.  From there, a scaling and the use of (\ref{eqn:homog}) yields lattices such that $\pl$ has various dynamics properties, for example a superattracting fixed point.  Since standard lattices have order 2 integer points without scaling, we use that in what follows.  For example, Figure \ref{fig:fourgroup} shows standard rectangular lattice where the attracting fixed point contains the integer critical value $1$ in the immediate attracting basin of an attracting fixed point.  A similar graphic for the standard rectangular square lattice is shown in Figure \ref{fig:Juliacurve} using $g_2=4$ and $g_3=0$ (see also Proposition \ref{prop:standardsquare}.

\begin{theorem}\label{thm:critvals} For integers $m > 0$, and $n \in \Z$ satisfying $|n| <  2 m$, set
\begin{equation}\label{eqn:critvals} g_2=4(m^2 + mn +n^2), \quad  g_3= -4 (m^2 n + m n^2). \end{equation} Then if $\LA = \LA(g_2,g_3)$, the critical values for
$\pl$ are $m,n,$ and $-m-n$;   there are three order 2 torsion points for $\pl$ occurring where $E(\R)$ crosses the x-axis: $(m,0), (n,0), (-m-n,0)$. \end{theorem}
\begin{proof} Substituting for $g_2$ and $g_3$, the polynomial (\ref{eqn:cubicpoly})
can be rewritten as $p(x)= 4 (m - t) (n - t) (m + n + t)$ under the hypotheses. \end{proof}

\noindent{\bf Examples.} {\bf (1)} \, When we set $n=1$ in Theorem \ref{thm:critvals}, we recover a standard lattice.  Versions of Theorem \ref{thm:critvals} for standard lattices appear in (\cite{hk3}, Thm 5.1) and(\cite{hm1}, Prop. 2.5).
{\bf (2)} \. An example of a standard lattice whose torsion group is the Klein four-group, comes from setting $g_2 = 8, g_3=-4.$  The elements of the torsion group, along with the attracting fixed point are shown in Figure \ref{fig:fourgroup}.  Two integer points are $c_1$, giving $A_1=(1,0) \in \widehat{F}(E)$ and $A_2= \orb = 2 A_1$, coming from  $2 c_1 = \lambda_1$.  There are two more integer points, $B_1=(0,2)$, $B_2=-B_1=(0,-2)$.  They come from non-real points in $\C/\Lambda$, points on the half-lattice line that maps to
$[e_2,e_3]\approx [-1.62, .62]$, with two crossings of the real axis per fundamental interval.

\begin{figure}
\begin{center}
{\includegraphics[width=.95\textwidth]{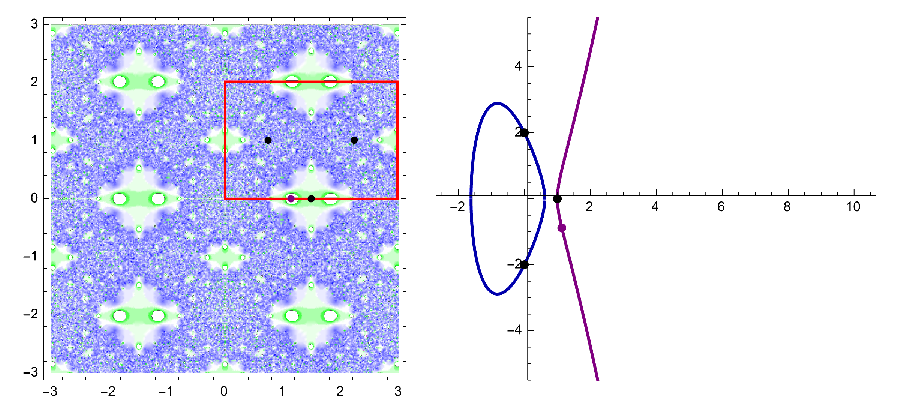}}
\caption{ $J(\pl)$, using $g_2=8, g_3=-4$ on the right. The attracting fixed point is marked in purple, the torsion point in its immediate attracting basin,
and two other integer points in black.  On the right, the associated elliptic curve with the same points marked. The other 3  points $(0,\pm 2)$ and $\orb$ in the torsion group are $J(\pl)$ since $0$ is a pole.} \label{fig:fourgroup}
\end{center}
 \end{figure}

 \begin{figure}
 \begin{center}
{\includegraphics[width=.95\textwidth]{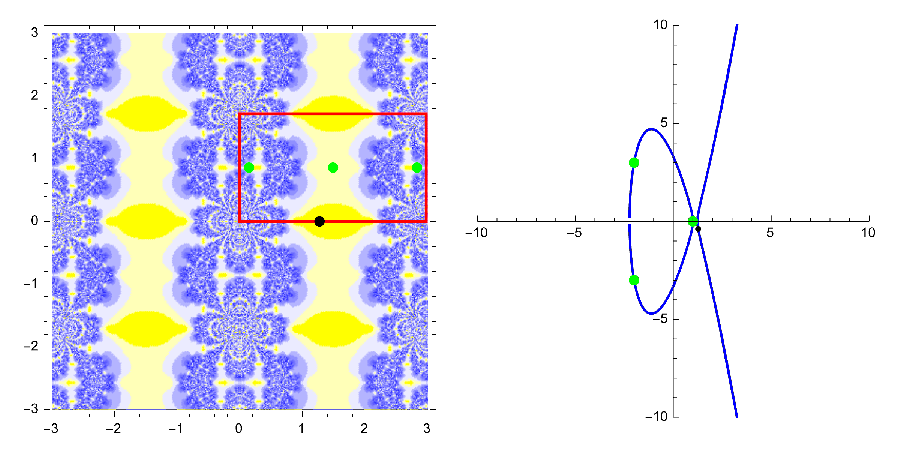}}
\caption{A toral band in $F(\pl)$, and $J(\pl)$ are shown at right, using $g_2=15, g_3=-11$. The attracting fixed point is black, and the torsion point in its immediate attracting basin,
and two other integer points are cyan.  On the right, the associated elliptic curve with the same points marked. The curve is nonsingular but the components come very close to each other.  All torsion points except $\orb$ are in $F(\pl)$.} \label{fig:toralband}
\end{center}
 \end{figure}

\begin{corollary}   For every $g_2 \in \N$, and integer $k$, $0\leq k< \sqrt{g_2}/3$,  setting $g_3 = -k g_2 + 4 k^3$, gives $(k,0)$ as an integer point for the real rectangular lattice $\LA(g_2,g_3).$  More generally, for $g_2 \in \Z$,  $k \in \Z$ setting $g_3 = -k g_2 + 4 k^3$ yields the integer point $(k,0)$ for the associated real lattice, as long as $g_2^3-27 g_3^2 \neq 0$.
\end{corollary}

\begin{proof}  To obtain a rectangular lattice, it is necessary that $g_2>0$.  Under the hypotheses, the three critical values of $\pl$ are
    $$k \quad {\rm{and}} \quad \frac {- k \pm \sqrt{g_2 - 3 k^2}}{2}.$$  This comes seeing that $p(k)=0$ in (\ref{eqn:cubicpoly} and then factoring
    $(t-k)$ out from (\ref{eqn:cubicpoly}).
For a rectangular lattice we must have 3 distinct real roots of $p(t)$, which is guaranteed by the condition on $k$.  In general the roots must be distinct as well, so  $\bigtriangleup \neq 0$, but the roots might not be real and the resulting lattice rhombic.  If $g_2=0$, then $g_3 = 4 k^3$ and the result still holds.
\end{proof}

\begin{proposition}\label{prop:standardsquare} If $g_2=4$ and $g_3 = 0$, (so m=1, n=-1 in (\ref{eqn:critvals})), then the there are torsion points $(1,0),(-1,0) \in \widehat{F}(E(\Q)))$, but $(0,0) \in \widehat{J}(E(\Q))$.
For $\LA(4,0)$, $E_{\rm tor}(\Q) = \Z_2  \oplus \Z_2$.  \end{proposition}

\begin{proof} It is known that the lattice $\LA(4,0)$ has side length
$\gamma \approx 2.62206$, where $\gamma$ denotes the lemniscate constant (see e.g., \cite{hm1}, Table 1).   It is the standard square lattice so the critical values are $1, -1$, and $0$.
We show there is an attracting fixed point between $e_1=1$, and $c_1 = \gamma/2$.  Since $\pl$ is decreasing on $I=(0,c_1)$, reaching a minimum value
of $1$ at $c_1$, we have:
$$  \pl(1) - 1 >0 \quad {\rm and} \quad  \pl(c_1) - c_1 < 0. $$  By the Intermediate Value theorem there is a point $p$ between $1$ and $c_1$ that is a fixed point.
To show it is attracting,  we estimate the derivative using the first 3 terms of the series expansion of $\pl'(t) $ at $t=0$.  Since $\pl(-z)=\pl(z)$, Both $(-1,0)$ and $(1,0)$ are in the Fatou set of $E(\Q)$.  Since
$(0,0)$ is a pole, it is in $\widehat{J}(E(\Q))$ (see Figure \ref{fig:Juliacurve}).  The last statement follows from the result for the Diophantine equation $\ds y^2 = x^3 - x$.
\end{proof}

\noindent{\bf Rectangular toral bands with integer points.}\quad We consider $\LA = \LA(11,-7)$; it is a standard, horizontal $2 \times 1$ lattice.  Therefore $e_1 = 1$, and $e_2 \approx .914$, and consequently $e_3 \approx -1.914$ (see \cite{hk2}, Section 10 where this example appears).

The critical point $[c_1]$ gives rise to the integer point $A= (1,0)$ on the elliptic curve. Since $[2 c_1]=[0] (\mod \LA)$, $2A = \orb$.
 Additionally, there is an attracting fixed point for $\pl$ at around $p \approx 1.18.$  The attracting basin of $p$ contains both $e_1$ and $e_2$ so we have a toral band
 containing an integer point.   Therefore $(0,1) \in \widehat{F}(E(\Q)$.  We have not found any other integer points for this example.

For $(g_2,g_3) =(15,-11)$, also a standard lattice with a toral band, the Julia set is shown in Figure \ref{fig:toralband}.  For this lattice there is an attracting fixed point $p \approx 1.27 > e_1$, which attracts 3 integer points:  $(1,0), (-2,3)$, and $(-2,-3)$.   Since $\LA$ is a standard lattice, the point $A=(1,0)$ is a torsion point.  Using
the point doubling formula, we easily compute that $B=(-2,3)$ is not.  One other noteworthy thing is that $A$ comes from $[c_3]$, located in the component of $E(\Q)$ not containing the fixed point.  In Figure \ref{fig:toralband} we see the two components of the curve are close but do not touch.

\begin{proposition}\label{prop:squaretorsion} For every lattice $\LA=\LA(g_2,0)$ (real, square) the point $A=(0,0)$ is always a torsion point of $E_{(g_2,0)}(\Q)$ of order 2.
The point comes from $c_1 = \lambda_1/2 \in \R$ if $g_2>0$ and $c_3 = (\lambda_1+\lambda_2)/2$ if $g_2<0$.  The point $A \in \widehat{J}(E_{(g_2,0)}(\Q)).$
If $g_2 = 4 k^2, \, k \in \N$, there are two more torsion points:  $B =(0,k), C=(0,-k)$. The points $B, C \in  \widehat{F}(E_{(g_2,0)}(\R))$ if $\pl$ has an attracting cycle.
 \end{proposition}
 \begin{proof} The first two statements follow from Theorem \ref{thm:critvals}.  If $g_2 = 4 k^2$, then by (\ref{eqn:pprimeprime}) and applying
 (\cite{hk1}, Lemma 6.5),  the critical point $c_1$ lies in the immediate attracting basin of the cycle, hence $B, C \in  \widehat{F}(E_{(g_2,0)}(\R))$.\end{proof}

\begin{figure}
\begin{center}
{\includegraphics[width=.95\textwidth]{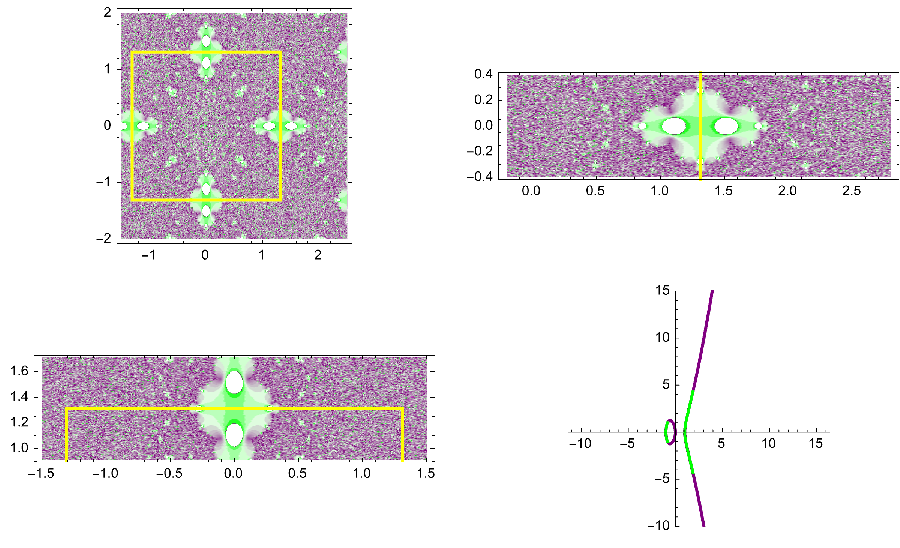}}
\caption{We show $J(\pl)$ for the real square lattice $\LA(4,0)$ (upper left), with its associated real elliptic curve (lower right).
In the upper right (near real axis)  and lower right (near a half lattice line) we show real-valued points in $\C$ that are mapped onto the elliptic curve.  We color points green if
they are the image of the green Fatou component, intersecting $\R$ or a horizontal half lattice line.} \label{fig:Juliacurve}
\end{center}
 \end{figure}

\section{Rhombic lattices with integer points}

The case of a rhombic square lattice was studied in \cite{hawkins0}, where it was shown that
for every rhombic square lattice $\LA$, $J(\pl) = \ci$.  Additionally we have that
the critical values are
\begin{equation} \label{eqn:rhombicsquare} e_1 = \alpha i, \, e_2 = -\alpha i, \alpha = \frac{\sqrt{|g_2|}}{2} >0, \, \, {\rm and} \, e_3 = 0.\end{equation}

 Dynamically, this means the following:  (1)  There is exactly one real critical value, $0=e_3$  corresponding to
$\pl(\lambda+i \lambda)/2)= \pl(c_3)$ giving the integer point $(0,0)$;   (2)$e_3 \in J(\pl)$.

We have therefore shown the following.
\begin{proposition} Elliptic curves of the form $E_{(g_2,0)}(\Q))$, $g_2<0$,  have no integer points in $\widehat{F}(E(\Q))$.
 \end{proposition}
 However, nontrivial torsion groups always occur by Proposition \ref{prop:squaretorsion}, and the values of $\pl$ at quarter lattice points were analyzed in \cite{hawkins0}  leading to the following result.

 \begin{proposition}\label{prop:rhomibicsq}   If $\Lambda = \Lambda(g_2,0)$ with $g_2 = -16 k^4, k\in \N$, then the elliptic curve $E_{\{g_2,0\}}(\Q)$ contains a torsion subgroup
isomorphic to $\Z_4$. \end{proposition}
 \begin{proof}  We write the lattice by $\Lambda = [a+ia, a-ia]$, and label the critical points by $c_1 = a/2 + i a/2, c_2= a/2 - i a/2$, and $c_3 = a$.
 Then the nonzero critical values are
 $ \pm 2 k^2 i$ by (\ref{eqn:rhombicsquare}).
 Applying (\cite{hawkins0}, Lemma 3.2) and  using (\ref{eqn:mainid}),
 $$ \pl (c_3/2) = \sqrt{-g_2/4} = 2 k^2,\quad {\rm and} \quad   \pl ' (c_3/2) = 8 k^3,$$
both integers.
 Therefore if we label $Q = (\pl(c_3/2),\pl'(c_3/2)) = (2 k^2, 8 k^3),$ we have
 $2Q = (0,0), 3Q=(2 k^2, -8 k^3),$ and $4Q = (\infty,\infty)$, corresponding to the identity element $\orb$. Since $Q$ generates a group of order 4, it is isomorphic to $\Z_4$.
 \end{proof}

\noindent {\bf Example.} The rhombic square lattice $\LA(-12,0)$ is not of the form in Proposition \ref{prop:rhomibicsq}, but it has many integer points.  One can verify directly from the equation
$\ds y^2 = 4 x^3 + 12x $ that points $$(x,y) = (1,\pm 4),(3,\pm 12),(12, \pm 84)$$ are all integer points.  However if $A = (1,4)$, we see that $2A$ has $\frac 14$ as its first coordinate, so the rank is at least 1.  In fact, the rank is known to be one and the torsion subgroup has exactly two elements in it: $(0,0)$ and $\orb.$

\subsubsection{Triangular lattices}

A real lattice $\Lambda$ is {\em triangular} if and only if $g_2=0$.  A fundamental region is two equilateral triangles, with a common side, the shorter diagonal of the rhombus being the same length as $\lambda_1$ and $\lambda_2$.   $\Lambda$ is generated
by 
$$\lambda_1= a (\sqrt{3} + i), \quad {and} \quad \lambda_2= a (\sqrt{3} -i),  \quad a>0.$$   A real triangular lattice is vertical if $g_3>0$,  and horizontal if $g_3<0$.
Writing as before, $c_1= \frac 12 \lambda_1 = b+ ia$,  $c_2= \frac 12 \lambda_2 = b -  ia$, when $\LA$ is horizontal, $e_3 <0$ and therefore $\pl$ has 2 real zeros
The Julia set of $\pl$ when $\LA$ is triangular is always connected (\cite{hk3}, Corollary 3.3), and studies have been done on the dynamics of $\pl$ in this case \cite{hk2}.

\begin{figure}
\begin{center}
{\includegraphics[width=.95\textwidth]{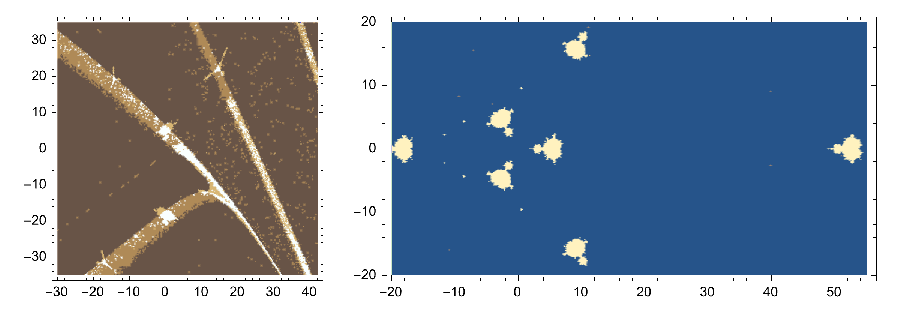}}
\caption{On the left is real $(g_2,g_3)$ space, where light colored points are pairs $(g_2,g_3)$ with attracting periodic orbits.  On the right, we show $g_3$ space, and look near the horizontal axis at the real integer values of $g_2$. These figures appeared over 20 years ago in \cite{hk3} to study purely dynamical properties, but can be compared with integer points for the corresponding elliptic curves.} \label{fig:triangle}
\end{center}
 \end{figure}

\begin{figure}
\begin{center}
{\includegraphics[width=.95\textwidth]{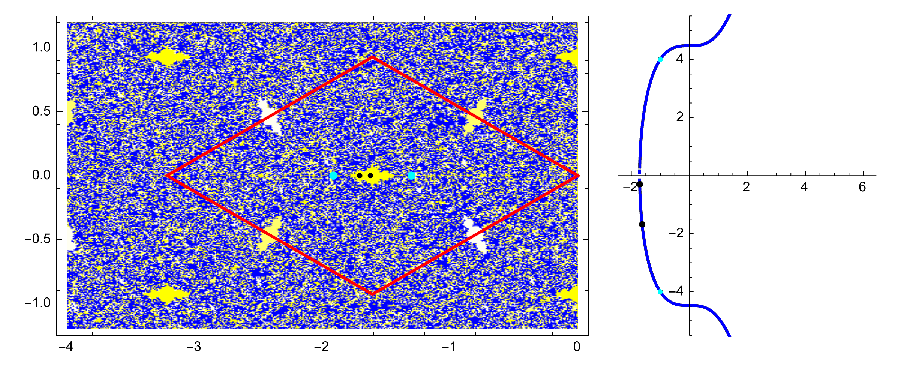}}
\caption{As suggested by the left side of Figure \ref{fig:triangle}, using $g_3=-20$ yields a triangular lattice for $\pl$ with an attracting period 2 orbit.
on each, the cyan points mark integer points and the black points show an attracting period 2 orbit.} \label{fig:per2}
\end{center}
 \end{figure}

To obtain triangular lattices with order 2 torsion points we exploit the symmetry in triangular lattices. An order 2 torsion point must correspond to a real critical point in $\C$ if $\Lambda$ is not rectangular. Setting $\rho=e^{2 \pi i/3}$, the critical values of  $\pl$ are $e_1$, $\rho e_1$, and $\rho^2 e_1$.   Moreover, by (\ref{es}), $ \ds e_1 e_2 e_3 = g_3/4.$  Therefore obtaining a 2-torsion point of the form $A=(k,0)$ follows naturally from these remarks.

\begin{lemma}\label{lem:triangle} Setting
$\LA = \LA(0,g_3)$ with $g_3 = 4 k^3$, $k \in \Z$, $k \neq 0$, gives an integer point $(k,0)$ associated to the real critical point;  $k$ can be positive or negative. \end{lemma}

Using a Mathematica program, we see that the integers
$$\{6, 7, 11, 13, 14, 20, 21, 23, 29, 32, 34, 39, 42, 45, 46, 47\}$$ are the positive integers less than $50$ for which the equation
$\ds  y^2 =x^3 + n$
has no integer solutions. Otherwise, there are solutions (see \cite{silvermantate}).  Multiplying the list by 4 and looking at Figures \ref{fig:triangle}, we obtain triangular lattices $\LA(g_2,g_3)$ for which $\pl$ has an attracting periodic orbit, and integer points.  We ask if $F(\pl) \neq \emptyset$ when are there integer points in $\widehat{F}(E_{(g_2,g_3)}(\Q))$.

An illustrative example is $(g_2,g_3) = (0,-20)$.  The associated equation is $y^2 = 4 x^2 + 20.$
The lattice $\LA(0,-20)$ has a rhombic fundamental region with side length approximately $1.86$.  The set $J(\pl)$ for this lattice  has a real attracting period 2 orbit near $e_3 \approx -1.71$, the period 2 orbit being closeby. The points in the attracting orbit are very near each other.  Calculations show there is a point $q_1 \approx -1.30$ such that $\pl(q_1)=-1$ and $\pl'(q_1) =4$.  We can show that $q_1$ is not in the immediate attracting basin of the period 2 orbit, but is mapped to it.  So $q_1 \in F(\pl)$, and  since integer points occur in pairs,
by setting $d_1= q_1-c_3$, and setting $q_2 = c_3 - d_1$, we see that $\pl(q_2) = -1$ while $\pl'(q_2) = -4.$  By symmetry of the Julia and Fatou sets, both integer points are in $F(\pl)$, and this gives $A=(-1,4), B=(-1,-4)$ with $A,B \in \widehat{F}(E(\Q))$.  Using (\ref{eqn:double}) we determine that $q_1$ does not have finite order, since $\pl(2 q_1)$ is not an integer.

The following theorem gives sufficient conditions for an order 3 torsion point for a triangular lattice.
\begin{theorem}  Suppose $\LA$ is a real triangular lattice with $g_3= - k^2$, $k \in \N$.  Then there are integer points of the form $(0,\pm k)$ and they form a torsion subgroup of order 3 (with $\orb$). If $g_3 = -4$, we also obtain an order 2 point $(1,0)$ and $E_{(0,-4)} = \Z_6$. \end{theorem}
\begin{proof}  Under the assumptions given, since $g_3 <0$, we have that $e_3 < 0$ and a fundamental region is a horizontal rhombus.  Therefore there exist two real points $t_1,t_2$ in each fundamental parallelogram
where $\pl(t_j)=0$.  At these values of $t_j$, the points $(0, \pm k)$ satisfy $\ds y^2 = 4\cdot 0 + k^2 = k^2$; by (\ref{eqn:mainid}), $\pl'(t_j)= \pm k$.  If $A$ is an integer point of
order 3, then $3A = \orb$ or equivalently $\pl(2 t) = -\pl(t) (\bmod \Lambda)$. By symmetry of the lattice, the zeros of $\pl$ for a triangular lattice are located in the centers of the equilateral triangles that make up a fundamental rhombic region, which,  in this case are real and located
at $\pm (2/3) c_3$.  Setting $A = (0, -k)$, $2A =(0,k)$ is at $(4/3) c_3 = - (2/3) c_3 (\bmod \LA)$, (see \cite{duv}, Chap. 2.21).  The point $3A = \orb$  since $(6/3) c_3$ is a pole.

If $g_3=-4$, then both the hypotheses of Lemma \ref{lem:triangle} and the theorem are satisfied, so we obtain a torsion subgroup of $\Z_6$.

\end{proof}

\subsubsection{Rhombic toral bands}
Using invariants $g_2=26, g_3=-26$, we claim that integer points occur simultaneously with toral bands for a rhombic lattice.  We make a few observations:  first, the associated lattice is rhombic since $\bigtriangleup(g_2,g_3)<0$, see also Figure \ref{fig:rhombictoral}. A similar example was studied in (\cite{hk3}, Example 3.6), a key difference being that by choosing close but non-integer $g_2$ (26.56) and $g_3$ (-26.26), the underlying lattice was real retangular, not rhombic.  However the dynamics remain the same while integer points appear.  We prove the following lemma.

\begin{lemma}  \label{lem:oppositeg2g3} Assume $\LA=\LA(g_2,g_3)$ with $g_3=-g_2$, with $g_2 \in \N$, $g_2 < 27$.  Then the points $A=(1,2), B=(1,-2)$ are integer points. \end{lemma}
\begin{proof}  We consider the Diophantine equation $y^2 = 4 x^3 - g_2 x - g_3$, and make the substitutions $x=1, g_3=-g_2$.  Then $y^2 = 4$, so $y=\pm2$.
Since $g_3 <0$, $e_3<0$ so there are two real points (in each fundamental region on $\R$) such that $\pl(t_j) = 1, j=1,2$.   If $t_1<t_2$, then
$\pl'(t_1)=-2$ and $\pl(t_2) = 2.$  \end{proof}

\begin{remark}   Integer points are typically not torsion points. They sometimes generate an infinite subgroup of $E(\Q)$.  \end{remark}

The Julia sets from Lemma \ref{lem:oppositeg2g3} vary widely for different integer values of $g_2$.  In Figure \ref{fig:rhombictoral} we show that choosing $g_2 = 26$ gives a toral band.  The reason for this is discussed for example, in \cite{hk3}.  The real critical value is located near $-2.95$, and there is an attracting fixed point
$t_0 \approx 1.39$,
which attracts the complex conjugate critical values.  This can only happen when we have an extreme horizontal rhombic lattice because all other graphs are piecewise concave up on $\R$.  When $g_2 = 26$, all 3 critical values eventually get mapped into the attracting basin of $t_0$.

\begin{figure}
\begin{center}
{\includegraphics[width=.95\textwidth]{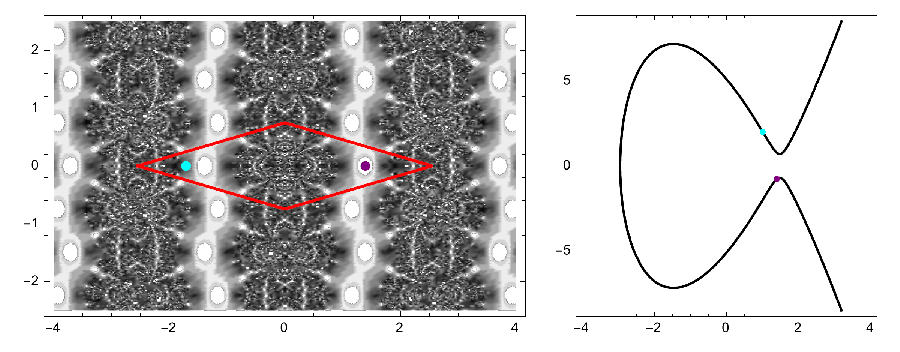}}
\caption{For $\Lambda(26,-26)$, $J(\pl)$ is on the left showing the attracting fixed point and an integer point.  At right, the associated real elliptic curve, where the attracting fixed point is shown along with
the integer point $(1,2)$ from Lemma \ref{lem:oppositeg2g3}.} \label{fig:rhombictoral}
\end{center}
 \end{figure}

\section{Dynamical remarks about some nondynamical examples}

From a paper in 2002 \cite{brown}, we see that fixing $g_2=4$, and letting $g_3 = - 4 m^2$ for
 integers $m \geq 1$, yields an interesting family of examples.
For all $m \in \N$ we have a horizontal extreme rhombic lattice, so the dynamics of $\pl$ when $g_2=4$ and $g_3 =- 4 m^2$ has been studied in (\cite{hawk26} and the references contained therein). For example, $J(\pl)$ is not a Cantor set, but might have a toral band. The arithmetic properties of equations of the form
 \begin{equation}\label{eqn:brown} y^2 = x^3 - x + m^2 \end{equation} are thoroughly studied in \cite{brown}, where is shown that each $m$ yields many integer points, but the torsion group is trivial.

 To illustrate, the cubic Diophantine equation in (\ref{eqn:brown}) using $m=2$
 corresponds to a horizontal extreme rhombic lattice with invariants $g_2 =4, g_3 = -16$.   By inspection we can ``guess'' a few integer solutions to the (\ref{eqn:brown}), namely
 $ \ds (x,y) = (0, \pm 2), (1, \pm 2), (-1, \pm 2), (4, \pm 8).$ Since $g_3<0$, the real critical value $e_3< 0.$
 The pair $(0, \pm 2)$ corresponds to two points, $t_1,t_2$ in one real fundamental period line segment of the (approximate) form $I=[-2 c_3, 0]$ since $e_3 \in I$ and also there is an attracting fixed point $p \approx -1.78 \in I$ as well.  In $I$, we find $t_1,t_2$ so that  $\pl(t_1) =0$ and $\pl'(t_1) = 4$, and   $\pl(t_2) =0$ and $\pl'(t_2) = -4$, as shown in Figure \ref{fig:special16}.  Both $t_2 \approx -2.31$ and $s_2=-t_1 \in [t_2]$ work.
  Figure \ref{fig:special16} shows the elliptic curve and $J(\pl)$ for $\LA(4, -16)$.  We mark (in cyan) the points $A_1,A_2=(-1,\pm 4),A_3,A_4=(0, \pm 4),A_5,A_6= (1,\pm 4)$ and the attracting real fixed point $t_0 \approx -1.78$, with $\pl'(t_0) \approx -.73$.  Applying formula  (\ref{eqn:double}), we see that  $2 A_j \in \Q$ but is not an integer; however, numerically, it appears that $A_j \in \widehat{F}(E_{(4,-16)}(\Q))$ for all $j$.

\begin{figure}
\begin{center}
{\includegraphics[width=.95\textwidth]{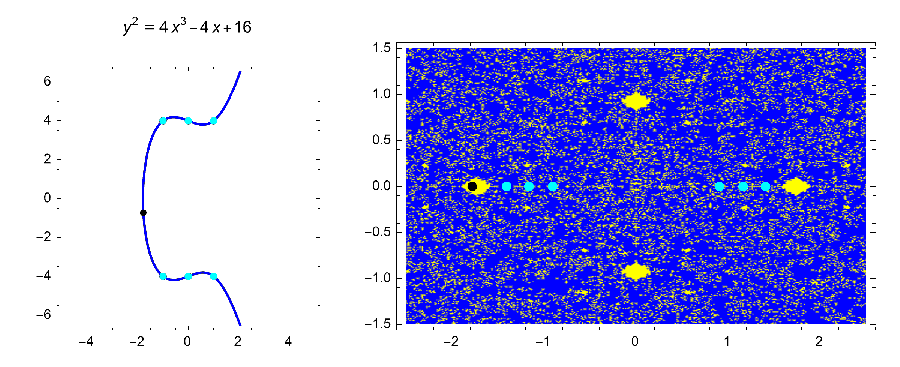}}
\caption{On the left we have the elliptic curve for $\LA(4,-16)$ and whose equation is shown. On the right the Julia set is blue, and the Fatou set is yellow.  In both graphics the attracting fixed point and some integer points are shown in cyan.} \label{fig:special16}
\end{center}
 \end{figure}

A computer was used to generate the next example that was already known to exist.
We present it as a lattice with the property that dividing the real line segment into 7 equal size pieces yields a torsion group that contains $\Z_7$.
 Setting $(g_2,g_3)= (172,664)$, the lattice associated to these invariants is an extreme horizontal rhombic lattice as shown in Figure \ref{fig:lattice7}, and the points that generate the integer torsion group are evenly spaced at intervals of $1/7$ along the segment of length $\approx  2.17338$.
 The acute angle that the two rhombus sides make is approximately (but likely not exactly) $\pi/4$.
 As we move from left to right along the points in Figure \ref{fig:lattice7}, and evaluate
 $\pl$ at each of these 7 points we obtain these values:  $\pl(t)= \{\infty, 11, 3, -5, -5, 3, 11\}.$  The value $\infty$ occurs at the origin and corresponds to the identity element  $\orb$.  The graphs of $\pl$ and $\pl'$, along with a few horizontal lines showing the integer values that occur, are shown in Figure \ref{fig:lattice7}.
 The corresponding values of $\pl'$ on the same set of points
  $ \ds \pl'(t)= \{\infty,-64, -16, -32, 32, 16, 64\}.$
This yields a set of 7 points, including the point at the pole at $0$ that gives $\orb$, that satisfy the cubic equation
  $ \ds y^2 = 4 x^3 -172 x + 664,$  (or it can be divided by 4 to obtain a monic  equation).

Because $g_2$ and $g_3$ are large, drawing $J(\pl)$ is difficult.  Estimates show the real critical value, $e_3$
is negative and large, $e_3 \approx -8$, so an attracting cycle would have at least one point close to $-8$.
While an attracting fixed point can be ruled out, there could be a cycle with a large period
attracting $e_3$.  The non-pole integer points occur in symmetric pairs about the (real) critical point $c_3$, so under iteration of $\pl$ will collapse
to at most 3 orbits in $\R$.  Further connections are the subject of future work by the author.

\begin{figure}
\begin{center}
{\includegraphics[width= .95 \textwidth]{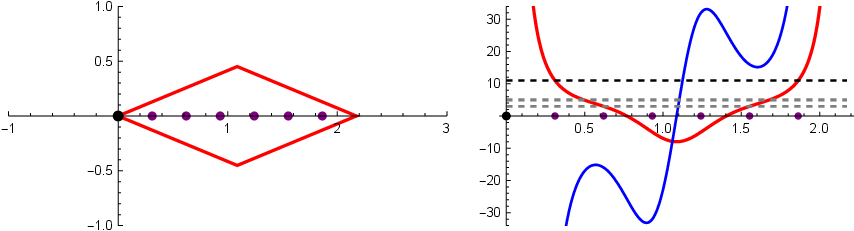}}
\caption{ Left, a fundamental region for the lattice
$\LA(172,664)$ with real points marked that have integer values for $\pl$ and $\pl'$.
Right, the graph of $\pl(t)=s$ in one periodic  interval in red, with $\pl'(t)=s$ in blue and the lines $s=3,5,11$.
Torsion points are marked in purple, with the identity located in black at $0$.}
\label{fig:lattice7}
\end{center}
\end{figure}

\bibliographystyle{amsplain}

\end{document}